\documentclass[12pt,reqno]{amsart} 

\usepackage{amsmath, amssymb} 
\usepackage{amscd}            
\usepackage{amsxtra}          
\usepackage{upref}            
\usepackage{xypic}
\usepackage[all]{xy}
\usepackage{amsfonts, amsthm}
\usepackage{graphicx}
\usepackage[mathscr]{eucal}

\theoremstyle{plain}

\newtheorem{theorem}{Theorem}

\newtheorem{lemma}[theorem]{Lemma}

\theoremstyle{remark}

\numberwithin{theorem}{section}
\numberwithin{equation}{section}

\newcommand{\A}{\mathbb A}
\newcommand{\C}{\mathbb C}
\newcommand{\Q}{\mathbb Q}
\newcommand{\F}{\mathbb F}
\newcommand{\R}{\mathbb R}
\newcommand{\G}{\mathbb G}
\newcommand{\Z}{\mathbb Z}

\title[Curves with isomorphic Jacobians]
{A refined notion of arithmetically equivalent number fields, and curves with isomorphic Jacobians}

\author{Dipendra Prasad}

\address{School of Mathematics, Tata Institute of Fundamental
Research, Colaba, Mumbai-400005, INDIA}
\email{dprasad@math.tifr.res.in}

\begin{document}

\begin{abstract}
 We construct examples of number fields which are not isomorphic but for which their
idele class groups are isomorphic. We also construct examples of projective algebraic curves which are
not isomorphic but for which their Jacobian varieties are isomorphic. Both are constructed using an 
example in group theory provided by Leonard Scott of a finite group $G$ and subgroups $H_1$ and $H_2$ 
which are not conjugate in $G$ but for which the $G$-module $\Z[G/H_1]$ is isomorphic to $\Z[G/H_2]$.
\end{abstract}

\maketitle

Two number fields $K_1$ and $K_2$ are said to be arithmetically equivalent 
if their zeta functions are the same. It is known that arithmetically equivalent number fields 
have the same degree over $\Q$, same infinity type, same discriminant, same roots of unity, but possibly different 
finite types, i.e., $K\otimes \Q_p$ may not be the same, and also 
class numbers and regulators might be different. Arithmetically equivalent number fields all arise from a simple group theoretic point of view which we now recall. 

Call a triple of finite groups $(G,H_1,H_2)$ with $H_1$ and $H_2$ subgroups of $G$, a Gassmann triple,  if 
$\Q[G/H_1]$ and  $\Q[G/H_2]$ 
are isomorphic as $G$-modules but $H_1$ and $H_2$ are not conjugate in $G$. 
Such examples exist in abundance, and one good source of them is $(G(\F_q),P_1(\F_q),P_2(\F_q))$
where $G$ is a reductive group over a finite field $\F_q$, with $P_1$ and $P_2$ non-conjugate 
parabolic subgroups in $G$ but for which their Levi subgroups are conjugate; the smallest such example is therefore for $G=SL_3(\F_2)$, 
a simple group of order 168 containing $P_1(\F_2)$ and $P_2(\F_2)$ as subgroups of index 7. It is 
useful to note that if $(G,H_1,H_2)$ is a Gassmann triple,  and if $N$ is a normal subgroup of $G$, then
$(G/N,H'_1,H'_2)$ is a Gassmann triple in $G/N$ where $H_i'$ is the image of $H_i$ in $G/N$.   

It is known that two number fields $K_1$ and $K_2$ have the same zeta functions if and only if 
they have the same Galois closure over $\Q$, with Galois group $G$, and are obtained as 
fixed fields of subgroups $H_1$ and $H_2$ of $G$ such that $(G,H_1,H_2)$ forms a Gassmann triple.

The refined notion of arithmetic equivalence that we discuss in this paper 
replaces the isomorphism between $\Q[G/H_1]$ and  $\Q[G/H_2]$ to one between
$\Z[G/H_1]$ and  $\Z[G/H_2]$. This
 implies closer relationship between 
the number fields involved than has been considered before, in particular, their class groups and  idele class groups are isomorphic. 

Given the analogy between class groups and the Jacobian of projective algebraic curves, it is natural 
that the same ideas give a general construction of curves with isomorphic Jacobians. This allows
as in Sunada's work (on isospectral but not isomorphic Riemann surfaces), construction of 
a general class of curves which are not isomorphic but whose Jacobians are. Apparently, the known examples
so far were only for  small genus by explicit constructions, such as by E. Howe for genus 2 and 3, and by 
C. Ciliberto  and G. van de Geer for genus 4.

A somewhat surprising example due to Leonard Scott in finite group theory acted as a catalyst to this work. We state it as a theorem.

\begin{theorem}There is 
a triple of finite groups $(G,H_1,H_2)$ with $H_1$ and $H_2$ subgroups of $G$ such that 
$\Z[G/H_1]$ and  $\Z[G/H_2]$ 
are isomorphic as $G$-modules, but $H_1$ and $H_2$ are not conjugate in $G$. In fact there is such an example 
for the group $G={\rm PSL}_2({\mathbb F}_{29})$, 
with  both $H_1$ and $H_2$ isomorphic to $A_5$ (and are conjugate in ${\rm PGL}_2({\mathbb F}_{29})$).
\end{theorem}

If we recall the (contravariant) duality between tori $T$ of dimension $n$ over a field $k$, and free abelian groups $X(T)$ of rank $n$ together 
with a representation of Gal$({\bar k}/k)$ on $X(T)$, then the above theorem allows us to construct tori $T_1$ and $T_2$
over $k$ (we are assuming existence of a Galois extension of 
$k$ with Galois group $G$; for the particular example mentioned in the theorem of Scott above, there is now a theorem due to 
D. Zywina constructing a Galois extension of $\Q$ with Galois 
group 
$G={\rm PSL}_2({\mathbb F}_p)$ for any prime $p$, in particular
$G={\rm PSL}_2({\mathbb F}_{29})$), such that $T_1(k) = K_1^\times$,
and $T_2(k) = K_2^\times$, for finite extensions $K_1,K_2$ of $k$, which has the property that although $T_1 \cong T_2$
as tori over $k$, in particular $K_1^\times \cong K_2^\times$ through an algebraic isomorphism, 
there is no isomorphism of fields
$K_1 \rightarrow K_2$.

Given the isomorphism of tori $T_1$ and $T_2$ over $k$, 
which we now assume is a number field,
 we get an isomorphism of adelic groups $T_1(\A_k)$ and $T_2(\A_k)$, 
taking $T_1(k)$ to $T_2(k)$, and hence  an 
isomorphism of the idele class groups:
$$C_{K_1}=\A_{K_1}^\times/K_1^\times \longrightarrow \A_{K_2}^\times/K_2^\times = C_{K_2},$$
which does not arise from an isomorphism of fields $K_1 \rightarrow K_2$. It would be interesting to know 
if an isomorphism of the idele class groups $C_{K_1}$ with $C_{K_2}$ arises only through the construction here,
i.e., through an isomorphism of $G$-modules, $\Z[G/H_1]$ and  $\Z[G/H_2]$.

Thus we conclude that the Neukirch-Uchida-Pop theorem according to which an isomorphism of  Gal$({\bar {\Q}}/K_1)$
with Gal$({\bar {\Q}}/K_2)$ forces an isomorphism of the fields $K_1$ and $K_2$ 
does not hold good for their maximal abelian quotients
Gal$^{ab}({\bar {\Q}}/K_1)$ and 
Gal$^{ab}({\bar {\Q}}/K_2)$. See the paper [2] of Gunther Cornelissen and Matilde Marcolli for some attempts in this direction.

The  isomorphism between  the idele class groups
$\A_{K_1}^\times/K_1^\times$ and $ \A_{K_2}^\times/K_2^\times $ allows us to define a natural identification of the 
Gr\"ossencharacters on $K_1$ with that on $K_2$ (depending on an isomorphism of integral representations 
$\Z[G/H_1]$ and  $\Z[G/H_2]$). 
However, the resulting identification 
$\chi_1 \longleftrightarrow \chi_2$ between Gr\"ossencharacters on $K_1$ with that on $K_2$ 
does {\it not} 
have the property
that 
$$L(\chi_1, s) = L(\chi_2,s),$$ 
which is the case  if the isomorphism of $T_1$ with $T_2$ came from a field isomorphism.

Note that the isomorphism 
between $\Z[G/H_1]$ and  $\Z[G/H_2]$ induces in particular an isomorphism between  
$\Q[G/H_1]$ and  $\Q[G/H_2]$ hence in particular, the zeta function of the number fields
$K_1$ and $K_2$ must be the same.
We recall that equality of zeta functions of $K_1$ and $K_2$ implies equality of their residues (at $s=1$), 
and therefore by the Dirichlet
class number formula, of the quantities,
$$\frac{h_1R_1}{\sqrt{|d_1|}} = \frac{h_2R_2}{\sqrt{|d_2|}}.$$
Further, using the functional equation for the zeta functions, which involves the discriminiants of the number fields involved,
we find that $d_1 = d_2$. So for number fields with the same zeta functions, the class number formula yields,
$${h_1R_1} = {h_2R_2}.$$

But it is also known that the individual numbers $h,R$ might vary. Unlike this, in our case, an isomorphism of
integral representations $\Z[G/H_1]$ and  $\Z[G/H_2]$ gives an isomorphism of the finite part of the 
idele groups, and therefore of their
maximal compact subgroups, implying that not only the class numbers but the class groups for $K_1$ and $K_2$ 
are the same. 

The isomorphism of tori $T_1$ and $T_2$ over $k$, gives in fact  an isomorphism of adelic groups $T_1(\A_L)$ and $T_2(\A_L)$, 
taking $T_1(L)$ to $T_2(L)$ for all extensions $L$ of $k$. Hence,  if we consider only 
those extensions $L$ of $k$ which are 
disjoint from the Galois closure $K$ of $K_1$ in $\bar{\Q}$, so that  
$L\otimes_{\Q} K_1 = LK_1$,  and $\A_L \otimes_{\Q} K_1 = \A_{LK_1}$, we have the    isomorphism of the idele class groups:
$$C_{K_1L}=\A_{K_1L}^\times/(K_1L)^\times \longrightarrow \A_{K_2L}^\times/(K_2L)^\times = C_{K_2L},$$
for all extensions $L$ of $k$ which are 
disjoint from $K$.

Denoting by $C\ell[N]$ the class group of a number field $N$, we get an isomorphism of the class group
of the number field  $LK_1$ with   the class group of the number field $LK_2$, and these isomorphisms 
are compatible for the natural map
from class group of $LK_1$ to the class group of $MK_1$ for all inclusion of number fields $L \rightarrow M$ (which are disjoint from $K$),
making the following commutative diagram of class groups:

\[
\xymatrix{
C\ell[LK_1] \ar@{ ->}[d] \ar[r]^{\cong} & C\ell[LK_2] \ar@{ ->}[d]_{}\\
C\ell[MK_1] \ar[r]^{\cong} & C\ell[MK_2]. 
} \]

The above conclusions on equality of the idele class groups, or of the class numbers, 
 can also be made for the Jacobian of algebraic curves which we take up in the next section. 

\section{Isomorphism of Jacobians}

\begin{theorem} Let $p: X\rightarrow Y$ 
be a Galois cover (not necessarily unramified) of algebraic curves over $\C$ with Galois
group $G$. Let $H_1$ and $H_2$ be two subgroups of $G$ such that  the integral representations $\Z[G/H_1]$ and  $\Z[G/H_2]$ of $G$ are
isomorphic. Define curves $X_1$ and $X_2$ over $\C$ such that their function fields are the $H_1$ and $H_2$ invariants of 
the function field of $X$. Then the Jacobian of $X_1$ is isomorphic to the Jacobian of $X_2$ as abelian varieties over 
$\C$. 
\end{theorem}

\begin{proof}  We will prove  isomorphism of the Jacobian varieties over $\C$, denoted to be $J(C)$ for any curve $C$ with
function field $\C(C)$. We will denote the function field of $X$ by $K$, $Y$ by $k$, 
and that of $X_1$ and $X_2$ by $K_1$ and $K_2$. 

By considerations earlier in this paper, the nonzero elements 
 $\C(X_1)^\times$ and $\C(X_2)^\times$ 
in the respective function fields are isomorphic through an algebraic isomorphism
over $\C(Y)$, where we consider  $\C(X_1)^\times$ and $\C(X_2)^\times$ as $\C(Y)$ rational points of the
tori $T_1= R_{K_1/k}{\G}_m$ and $T_2=R_{K_2/k}{\G}_m$ for $K_1= \C(X_1)$, $K_2=\C(X_2)$, and $k = \C(Y)$.

An algebraic homomorphism  $\phi: \C(X_1)^\times  \rightarrow \C(X_2)^\times$
 allows one to construct 
---as we shall do presently--- a group homomorphism 
from  the group of zero cycles on $X_1$ to the group of zero cycles on $X_2$    taking cycles of degree zero on $X_1$ 
to cycles
of degree zero on $X_2$, and taking the divisor on $X_1$ associated to a function $f$ on $X_1$ to the divisor associated
to the function $\phi(f)$ on $X_2$. 

We recall 
what the isomorphism of tori $T_1$ and $T_2$ over $k$ really means concretely. For this, 
let $\sigma_i$ for $i=1,\cdots, n$, for $n = [K_1:k]$ denote the distinct
embeddings of $K_1$ inside $K$ (over $k$). Then any algebraic homomorphism from $T_1$ to $K^\times = R_{K/k}(\G_m)$ is of the form
  $$x\longrightarrow \prod _{i} \sigma_i(x)^{n_i},$$
for certain integers $n_i$; for an appropriate choice of $n_i$, the image of this homomorphism from 
$T_1$ to $K^\times = R_{K/k}(\G_m)$ lands inside $K_2^\times = R_{K_2/k}(\G_m) \subset K^\times = R_{K/k}(\G_m)$,
and then for a more specific choice of $n_i$, the corresponding  homomorphism from 
$T_1$ to $T_2= K_2^\times = R_{K_2/k}(\G_m) \subset K^\times = R_{K/k}(\G_m)$ is an isomorphism of tori from $T_1$ to $T_2$.

Now for a point $P \in X_1$ we shall define a zero divisor $\phi(P)$ on $X_2$. For this take any function $f$ on $X_1$
with a simple zero at $P$ and no other zeros or poles  at any of the  points of $X_1$ which are 
in the image of $G$-translates 
of points  of $X$ lying over $P$ in $X_1$.  
Define $\phi(P) = {\rm div}_{ \langle P \rangle}(\phi(f))$ where ${\rm div}_{\langle P \rangle }(\phi(f))$
denotes the part of the divisor of $\phi(f)$  which is supported on
the image in $X_2$ of $G$-translates 
of points  of $X$ lying over $P$ in $X_1$.  
  
From the definition of an algebraic homomorphism     
  $\phi: \C(X_1)^\times  \rightarrow \C(X_2)^\times$, we find that $P\rightarrow \phi(P)$  is a well defined map from divisors in $X_1$ to divisors in $X_2$, i.e., it is independent of
the choice of the function $f$ made above. (Eventually it just means that the map $x\longrightarrow \prod _{i} \sigma_i(x)^{n_i},$ takes functions on $X_1$ which have no zeroes or poles on $\langle P \rangle_{X_1} $ (which is 
the set in $X_1$ obtained by taking the 
image of $G$-translates of an inverse image of $P$ in $X$)  to functions on $X_2$ which have no zeros or poles on 
$\langle P \rangle_{X_2}$ a set in
$X_2$ similarly defined, i.e., $\langle P \rangle_{X_2}$ is the set in $X_2$ obtained by taking the 
image of $G$-translates of an inverse image of $P$ in $X$.

From a homomorphism of zero cycles on $X_1$ of degree zero to zero cycles of degree zero on $X_2$, which preserves
principal divisors constructed above, we get an `abstract' homomorphism from $J(X_1)$ to $J(X_2)$, 
which we need to prove is algebraic. We are not sure if algebraicity follows
 from quoting a standard theorem in the subject, so we give
a detailed proof. Along the way, we 
will also prove that the map $P\rightarrow \phi(P)$ takes 
zero cycles on $X_1$ of degree zero to zero cycles of degree zero on $X_2$.

Note that corresponding to the commutative diagram of $G$-modules,

\begin{equation} 
\begin{gathered} 
\xymatrix{ & \Z[G]  \ar[ld] \ar[rd]^{\phi'}& \\ 
\Z[G/H_2]   \ar[rr]^\phi &  & \Z[G/H_1]  } 
\end{gathered} 
\end{equation} 
where the map $\phi'$ is defined by the commutativity of this diagram, and the map from
$\Z[G]$ into $\Z[G/H_2]$ is the natural one taking the identity element of $\Z[G]$ to the identity coset of $G/H_2$, we have the corresponding diagram of tori:

\begin{equation} 
\begin{gathered} 
\xymatrix{ & \C(X)^\times  & \\ 
\C(X_1)^\times  \ar[rr]^{\phi} \ar[ru]^{\phi'} &  & \C(X_2)^\times \ar@{^{(}->}[lu]  
} 
\end{gathered} 
\end{equation} 
and the diagram of the Jacobian varieties:

\begin{equation}
\begin{gathered} 
\xymatrix{ 
& J(X)  & \\
J(X_1)  \ar[rr]^{\phi} \ar[ru]^{\phi'} & & J(X_2). \ar[lu]   
} 
\end{gathered} 
\end{equation} 

Here the mapping from the Jacobian $J(X_2)$ to $J(X)$ defined by the inclusion of fields
$\C(X_2) \hookrightarrow \C(X)$ 
is clearly algebraic. The mapping from the Jacobian $J(X_1)$ to $J(X)$ defined 
through a mapping of invertible elements of the fields $\psi: \C(X_1)^\times \rightarrow \C(X)^\times$ by
  $$x\longrightarrow \prod _{i} \sigma_i(x)^{n_i},$$
is also algebraic. This follows because of the way the maps are defined on the Jacobian variety: the map
from  the Jacobian variety of $X_1$ to the Jacobian variety of $X$ 
defined by $\psi: \C(X_1)^\times \rightarrow \C(X)^\times$ is a sum of $n_i$ multiples of the maps defined by  
$x\longrightarrow  \sigma_i(x);$ but these maps are now given by embedding of fields, so by the standard
theory (i.e., the functorial nature of the Jacobian variety), the corresponding map from  the Jacobian variety of $X_1$ to the Jacobian variety of $X$ is indeed algebraic.
(This argument also  proves that the map $P\rightarrow \psi(P)$ takes 
zero cycles on $X_1$ of degree zero to zero cycles of degree zero on $X$, and then 
it also follows that the map $P\rightarrow \phi(P)$ takes 
zero cycles on $X_1$ of degree zero to zero cycles of degree zero on $X_2$.)

At this point we have proved that in the diagram 1.3 above, $\phi': J(X_1) \rightarrow J(X)$,  as well as the  
natural map $\iota: J(X_2) \rightarrow J(X)$  are algebraic. This allows us to 
prove from generalities below that the mapping $\phi: J(X_1) \rightarrow J(X_2)$ is algebraic. 
Recall that the mapping from $J(X_1)$ and $J(X_2)$ to $J(X)$ 
are finite maps onto their images.

\begin{lemma} (a) If a morphism of complex algebraic varieties $f: X\rightarrow Y$ lands inside a closed subvariety
$Z$ of $Y$, then the corresponding set theoretic mapping from $X$ to $Z$ is algebraic.  

(b) If an `abstract' homomorphism of Abelian varieties over $\C$, $f: X\rightarrow Y$ 
becomes algebraic after
an isogeny of Abelian varieties $Y\longrightarrow Y'$, i.e., the composed map, $f': X\rightarrow Y'$ is algebraic,  
then  $f: X\rightarrow Y$ is algebraic.  
\begin{equation}
\xymatrix{
 & Y \ar[d] \\
X\ar[ru]^{f} \ar[r]^{f'} & Y'
}
\end{equation}

\end{lemma}
\begin{proof} Part $(a)$ of the lemma is rather standard. For part $(b)$, it suffices by part $(a)$ to assume that
  $f': X\rightarrow Y'$ is a surjective mapping of Abelian varieties, and in fact an isogeny using the fact that 
there are no nonzero  abstract homomorphism from an abelian variety over $\C$ to a finite abelian group. The same fact goes into the proof of the assertion of the part $(b)$ of the Lemma.
\end{proof}
\end{proof}

\vspace{4mm} 
\noindent{\bf Remark 1:} 
Under some general conditions one can assert that the 
curves $X_1$ and $X_2$ are not isomorphic. Sunada does so using transcendental methods by 
appealing to the existence of a curve $Y$ which is uniformized by a discrete subgroup 
$\Gamma \subset {\rm PSL}_2(\R)$ 
whose commensurator in ${\rm PSL}_2(\R)$ 
is $\Gamma$. (This uses a theorem of Margulis, 
according to which for non-arithmetic discrete groups $\Gamma$, the commensurator of $\Gamma$ contains $\Gamma$ as a subgroup of finite index, and a theorem of L. Greenberg according to which
 most discrete subgroups giving rise to a compact Riemann surface of genus $g > 2$ are maximal). This will then give examples of curves 
which are non-isomorphic but whose Jacobians are isomorphic. It is not clear if these transcendental methods 
can be replaced by more algebraic ones so that they work for $\bar{\F}_p$.

\vspace{4mm}

\noindent{\bf Remark 2:} 
In the construction here 
of nonisomorphic number fields $K_1$ and $K_2$ 
with the property that there is an isomorphism between $\A_{K_1}^\times \longrightarrow \A_{K_2}^\times$ taking $K_1^\times$ isomorphically to 
$K_2^\times$, we do not know 
if there is an isomorphism of topological rings
between $\A_{K_1}$ and $\A_{K_2}$ (which by assumption cannot take $K_1$ to $K_2$).

\vspace{4mm}

\noindent{\bf Remark 3:} For a finite Galois cover $p: X\rightarrow Y$ of projective algebraic curves with Galois group $G$, it is not true that $({\rm Pic}^0(X)^G)^0 = {\rm Pic}^0(Y)$ 
since the natural mapping from ${\rm Pic}^0(Y)$ to ${\rm Pic}^0(X)$ may have a kernel. 
Similarly, it is not true that ${\rm Pic}^0(X)_G$, the maximal qotient 
of ${\rm Pic}^0(X)$ 
on which $G$ operates trivially is  ${\rm Pic}^0(Y)$ (because the kernel of the map from ${\rm Pic}^0(X)$ to ${\rm Pic}^0(X)_G$ is connected).   
If either of these were true, 
then the above theorem would be true in a straightforward way.

\vspace{1cm}
{\bf Acknowledgement:} The author thanks Prof. Maneesh Thakur for asking a question which started it all: if an isomorphism of tori
$K_1^\times$ with $K_2^\times$ comes from an isomorphism of $K_1$ and $K_2$. The corresponding question on integral representations
was answered in successive approximation by Bart de Smit, Wolfgang Kimmerle, and finally by Leonard Scott who had the counter-example. I am thankful to them all. This paper is based upon work supported by the National Science Foundation
under grant No. 0932078000 while the author was in residence at the MSRI, Berkeley during the Fall semester
of 2014.

\vspace{1cm}

{\bf \Large Bibliography}

\vspace{1cm}
[1] C. Ciliberto, G. van der Geer: {\it Non-isomorphic curves of genus four with isomorphic (non-polarized) Jacobians.}
Contemporary Math series  of the AMS, vol. 162 (1994) 129-133. 

[2] G. Cornelissen,  M. Marcolli: {\it Quantum Statistical Mechanics, L-series and Anabelian Geometry}, arXiv:1009.0736.

[2] E. Howe: {\it Infinite families of pairs of curves over Q with isomorphic Jacobians,} J. London Math. Soc. 72 (2005) 327–350,

[3] Leonard L. Scott:
{\it Integral equivalence of permutation representations.} Group theory (Granville, OH, 1992), 262–274, World Sci. Publ., River Edge, NJ, 1993. 

[4] T. Sunada: {\it Riemannian coverings and isospectral manifolds}, Annals of Math., vol. 121 (1985) 169-186.

\end{document}